\documentclass[10pt]{article}
\font\smallit=cmti10
\font\smalltt=cmtt10

\usepackage{amssymb,latexsym,amsmath,epsfig} %% Add other packages as necessary
\usepackage{hyperref}
\usepackage{amsthm}

\makeatletter

\renewcommand\section{\@startsection {section}{1}{\z@}
{-30pt \@plus -1ex \@minus -.2ex}
{2.3ex \@plus.2ex}
{\normalfont\normalsize\bfseries}}

\renewcommand\subsection{\@startsection{subsection}{2}{\z@}
{-3.25ex\@plus -1ex \@minus -.2ex}
{1.5ex \@plus .2ex}
{\normalfont\normalsize\bfseries}}

\renewcommand{\@seccntformat}[1]{\csname the#1\endcsname. }

\makeatother

\begin{document}

\begin{center}
\uppercase{\bf On the Sum of Reciprocals of Amicable Numbers}
\vskip 20pt
{\bf Jonathan Bayless}\\
{\smallit Department of Mathematics, Husson University, Bangor, Maine 04401, USA}\\
{\tt BaylessJ@husson.edu}\\
\vskip 10pt
{\bf Dominic Klyve}\\
{\smallit Department of Mathematics, Central Washington University, Ellensburg, WA 98926, USA}\\
{\tt klyved@cwu.edu}\\
\end{center}
\vskip 30pt

\centerline{\smallit Received: , Revised: , Accepted: , Published: } % We will fill in the dates
\vskip 30pt 

\newtheorem{thm}{Theorem}[section]
\newtheorem{lem}[thm]{Lemma}
\newtheorem{conj}[thm]{Conjecture}

\centerline{\bf Abstract}

\noindent
Two numbers $m$ and $n$ are considered amicable if the sum of their proper divisors, $s(n)$ and $s(m)$, satisfy $s(n) = m$ and $s(m) = n$.  In 1981, Pomerance showed that the sum of the reciprocals of all such numbers, $P$, is a constant.  We obtain both a lower and an upper bound on the value of $P$.

\pagestyle{myheadings}
\markright{\smalltt INTEGERS: 10 (2010)\hfill}
\thispagestyle{empty}
\baselineskip=12.875pt
\setcounter{page}{1} %%%%%%%%%%% WE WILL ADJUST PAGE NUMBER
\vskip 30pt 

\section{Introduction} \label{sec0}

Since at least the time of the Ancient Greeks, amicable numbers have enjoyed the attention of mathematicians.  Let $s(n)$ denote the sum of the proper divisors of $n$, that is, $s(n) = \sigma(n) - n$, where $\sigma(n)$ denotes Euler's sum-of-divisors function.  Then an \emph{amicable pair} is a pair of distinct integers $m,n$, such that $s(n) = m$ and $s(m) = n$.  We will also refer to any integer which is a member of an amicable pair as an \emph{amicable number}.  The smallest amicable pair, $(220,284),$  was known to Pythagoras c.~500 BCE.  The study of amicable pairs was a topic arising often in Medieval Islam; as early as the 9th Century, Th\={a}bit had discovered three pairs, including (17296, 18416) -- a pair which was rediscovered independently by Borho, Ibn al-Bann\={a}', Kamaladdin F\={a}ris\={\i}, and Pierre de Fermat \cite{GPR}. In the 18th century, Euler \cite{E100} famously advanced the theory of amicable numbers by giving a table of 30 new pairs in 1747 (one of which is, in fact, an error -- see \cite{Sand}), and a larger table of 61 pairs, together with a method of generating them, in 1750 \cite{E152}.

Today, although much is known about amicable numbers (and their less popular friends, sociable numbers), there is still a lot that we don't know.  For example, although the best upper bound on their density was given by Pomerance \cite{Pom2} in 1981 (see \eqref{eq1} below), no known lower bound exists -- indeed, it has not been proven that there are infinitely many.  Considerable work has been done on questions such as the properties of even-even pairs, odd-odd pairs, pairs for which each number has exactly two prime factors not contained in its pair, and many more complex variations.  For a nice survey on amicable numbers, see \cite{GPR}.  One interesting fact, which motivated the present work, is that the sum of the reciprocals of the amicable numbers converges, that is,
\begin{equation} \label{eqP}
	\sum_{a \textrm{ amicable}} \frac{1}{a} = P < \infty.
\end{equation}

This is a consequence of the bound shown 28 years ago by Pomerance \cite{Pom2}, but to date no bounds have been given on the value of this sum.  In this paper, we provide an upper bound and a lower bound for the value of this sum.

\subsection{The distribution of amicable numbers}

Following \cite{Pom2}, we define $A(x)$ to be the count of amicable numbers not greater than $x$.  In 1955, Erd\H os \cite{Erdos} showed that the amicable numbers have density zero. Pomerance showed \cite{Pom1}
\begin{displaymath}
	A(x) \le x \exp\left( -c \sqrt{\log \log \log x \log \log \log \log x} \right)
\end{displaymath}
for a positive constant $c$ and all sufficiently large $x$.  Then, in 1981, Pomerance \cite{Pom2} improved this result to
\begin{equation} \label{eq1}
	A(x) \le C \frac{x}{\exp\left( \left( \log x \right)^{1/3} \right)}
\end{equation}
for some constant $C$ and all sufficiently large $x$.  From this and partial summation, it is clear that we have \eqref{eqP}.
 Since Pomerance was the first to show this sum to converge, we refer to its value, $P$, as \textit{Pomerance's constant}.

Having established that $P$ is finite, it is natural to ask about its value.  We prove the following theorem.
\begin{thm} \label{Thm1}
Pomerance's Constant, $P$, the sum of the the reciprocals of the amicable numbers, can be bounded as
$$
.0119841556 < P < 6.56 \times 10^8.
$$
\end{thm}
It should be noted that the upper bound in this result can be improved by a more careful choice of the functions in Table \ref{FunctionsTable} on page \pageref{FunctionsTable}.  It is unclear how to fully optimize this argument, but the authors have been able to show $P < 3.4 \times 10^6$.

To establish Theorem \ref{Thm1}, we split the sum defining $P$ into three parts:
\begin{equation} \label{split}
	\sum_{\textrm{amicable } n} \frac{1}{n} = \mathop{\sum_{\textrm{amicable } n}}_{n \le 10^{14}} \frac{1}{n} + \mathop{\sum_{\textrm{amicable } n}}_{10^{14} < n \le \exp\left(10^6\right)} \frac{1}{n} + \mathop{\sum_{\textrm{amicable } n}}_{n > \exp\left(10^6\right)} \frac{1}{n}.
\end{equation}
We evaluate the first sum directly in section \ref{sec1} to establish the lower bound for $P$.  In section \ref{sec3}, we bound the second sum.  We then modify the argument in \cite{Pom2} to address the final sum above in section \ref{sec4}.

Throughout, $p$, $q$, and $r$ will denote primes, $P(n)$ the largest prime factor of $n$, and $\psi(x,y)$ the number of $y$-smooth numbers up to $x$, that is, the size of the set
\begin{displaymath}
	\{n \le x | P(n) \le y \}.
\end{displaymath}
Finally, $a$ and $b$ will represent amicable numbers.

We will repeatedly make use of the following functions.  These will be redefined at the appropriate place, but the reader can use this table as a helpful reference.

\begin{table}[htb]%
\begin{center}
\begin{tabular}{rl}
$x_0$ &= $\exp(10^6)$\\
$y_0$ &= $\exp(26000)$\\
$c$ &= $c(x) = 1 - \left(\log x\right)^{-1/6}\left(\log \log x\right)^{-1}$\\
$c_0$ &= $c(10^6) > 0.99276$ \\
$\sigma$ &= $\sigma(y) = 1 - 1/(2\log y)$\\
$\ell$ &= $\ell(x) = \exp\left( \left(\log x \right)^{1/6} \right)$\\
$L$ &= $L(x) = \exp\left( 0.1882 \left(\log x\right)^{2/3} \log \log x \right)$
\end{tabular}
\end{center}
\caption{A table of functions used in the proof of Theorem \ref{Thm1}.}
\label{FunctionsTable}
\end{table}

\section{A Lower Bound -- Some Reciprocal Sums} \label{sec1}

Determining a lower bound on $P$ is straightforward -- we need only sum the reciprocals of any subset of the set of amicable numbers to find one.
Let $P_j$ be the sum of all integers not greater than $10^j$ which are members of an amicable pair, i.e.,
$$
P_j = \sum_{\substack{a \le 10^j \\ a \text{ amicable}}} \frac{1}{a}.
$$

The current record for exhaustive searches for amicable numbers is $10^{14}$ \cite{APweb}. In Table \ref{SmallSumTable}, we give values for $P_j$ for various powers of 10 up to this bound.

\begin{table}%
\begin{center}
\begin{tabular}{r|l}
$j$ &	$P_{j}$ \\
\hline
$1$ &	$0$\\
$2$	& $0$\\
$3$	& $0.0080665813060179257$\\
$4$	& $0.0111577261442474466$\\
$5$	& $0.0117423756996823562$\\
$6$	& $0.0119304720866743157$\\
$7$ & $0.0119714208511438135$\\
$8$ & $0.0119812212551025145$\\
$9$	& $0.0119834313702743716$\\
$10$ & $0.0119839922963130553$\\
$11$ & $0.0119841199294457703$\\
$12$ & $0.0119841486963721084$\\
$13$ & $0.0119841542458770555$\\
$14$ & $0.0119841556796931142$\\
\end{tabular}
\end{center}
\caption{Sums of reciprocals of amicable numbers to $10^j$.}
\label{SmallSumTable}
\end{table}

From this table, we immediately have
\begin{equation} \label{P1}
P > P_{14} > .0119841556796931142.
\end{equation}

We note that this series seems to converge rather quickly.  Indeed, for any $j \le 13$ with $P_j > 0$, the difference between between $P_{14}$ and $P_{j}$ is less than the difference between $P_{j}$ and $P_{j-1}$.  With this observation, we conjecture that the true value of $P$ can be estimated by
$$
|P - P_{14}| < P_{14} - P_{13} \approx .0000000006338,
$$
and thus:
\begin{conj}
Pomerance's constant, $P$, satisfies $$P < .0119841563134.$$
\end{conj}
This, however, is merely conjecture.  We turn now to establish the upper bound of Theorem \ref{Thm1}.

\section{The Middle Range} \label{sec3}

The amicable numbers in the range $[10^{14}, \exp(10^6)]$ are too large to be found explicitly, and too small to be amenable to the results we make use of later in this paper.  While there are some ways to restrict the sum over these numbers, their final contribution to Pomerance's constant is small, and we here simply make use of the trivial bound:
$$
\sum_{10^{14} < n \le \exp(10^6)} \frac{1}{n} < 10^6.
$$

\section{Preliminaries} \label{sec2}

We will use a number of explicit formulas of prime functions.  For instance, we will use the fact (see \cite[Theorem 5]{RS}) that
\begin{displaymath}
\sum_{p \le x} \frac{1}{p} \le \log \log x + B + \frac{1}{2 \left(\log x\right)^2}
\end{displaymath}
for $x \ge 286$, where $B = .26149721\ldots$.  Recall that we have chosen $x_0 = \exp(10^6)$.  This gives that, for $x \ge x_0$,
\begin{equation}
\label{eq3}
\sum_{p \le x} \frac{1}{p} \le \log \log x + .2615.
\end{equation}

We will use a bound on the sum of reciprocals of primes in a particular residue class modulo $p$, where $p > 10^{14}$.

\begin{lem} \label{myLem}
For any value of $y \ge 10^{14}$ and any prime $p \ge 10^{14}$, the following holds uniformly:
\begin{displaymath}
\mathop{\sum_{q \le y}}_{q \equiv -1 \kern-5pt\mod p} \frac{1}{q} \le \frac{4 + 3 \log \log y}{p}.
\end{displaymath}
\end{lem}
\begin{proof}
We begin with the Brun-Titchmarsh inequality of \cite{IBT};  namely, for coprime integers $k$ and $n$, the number of primes $q \le y$ with $q \equiv k \kern-5pt\mod n$, denoted $\pi(y;n,k)$, satisfies
\begin{displaymath}
\pi(y;n,k) \le \frac{2y}{\varphi(n) \log (y/n)} ~\hbox{ for }~y>n.
\end{displaymath}
For prime $p \ge 10^{14}$ and $(k, p) = 1$, we can use the fact that $\varphi(p) = p-1$ to see that
\begin{equation} \label{BTIneq}
\pi(y; p, k) \le \frac{2 y}{\varphi(p) \log (y/p)} \le \frac{2.0001 y}{p \log (y/p)}.
\end{equation}
Note that the first prime $q \equiv -1 \kern-5pt\mod p$ is at least $2 p -1$, so we need only consider $p \le \frac{y+1}{2}$.  Thus, $\frac{y}{p} \ge 2 - \frac{1}{p} \ge 1.999999$ for $p \ge 10^{14}$.  We use partial summation to obtain
\begin{displaymath}
\begin{aligned}
\mathop{\sum_{q \le y}}_{q \equiv -1 \kern-5pt\mod p} \frac{1}{q} &= \frac{\pi(y; p, -1)}{y} + \int_{2p-1}^y \frac{\pi(t; p, -1)}{t^2} ~dt \\
&\le \frac{2.0001}{p \log (y/p)} + \int_{2p-1}^y \frac{2.0001}{p \, t \log (t/p)} ~ dt \\
&= \frac{2.0001}{p \log (y/p)} + \frac{2.0001}{p} \int_{2p-1}^y \frac{dt}{t\log (t/p)} \\
&= \frac{2.0001}{p \log (y/p)} + \frac{2.0001}{p} \cdot \log \log (t/p) \bigg|_{2p-1}^y \\
&= \frac{2.0001}{p}\left( \frac{1}{\log (y/p)} + \log \log (y/p) - \log \log \left(2 - \frac{1}{p} \right) \right).
\end{aligned}
\end{displaymath}
Since $\frac{1}{\log (y/p)} - \log \log \left(2 - \frac{1}{p} \right) < 1.8093$ and $\log \log (y/p) \le \log \log y$, we see that

\begin{displaymath}
\mathop{\sum_{q \le y}}_{q \equiv -1 \kern-5pt\mod p} \frac{1}{q} \le \frac{4 + 3 \log \log y}{p},
\end{displaymath}
which proves the lemma.
\end{proof}

We will need the following lemma to bound the number of amicable numbers up to $x$ which are also $y$-smooth for some $y \ge y_0 = \exp\left(26000\right)$.  Recall that $\sigma = 1 - 1/(2\log y)$.

\begin{lem}\label{lem_bigproduct}
For $y \ge y_0$ and $\sigma$ as above, we have
$$
\prod_{p \le y} \left( 1 - \frac{1}{p^\sigma}\right)^{-1} < 7.6515\log y.
$$
\end{lem}

Before proving Lemma \ref{lem_bigproduct}, we establish a few necessary lemmas.

\begin{lem}\label{lem_logpoverp}
For $y \ge y_0$,
$$
\sum_{p \le y} \frac{\log p}{p} \le \log y - 1.3325.
$$
\end{lem}

\begin{proof}
We recall a result from Rosser and Schoenfeld (see \cite[3.23]{RS}).
Let $E = - \gamma - \sum_p (\log p)/p^n < -1.332582275$.  For $y \ge 319$,
$$
\sum_{p \le y} \frac{\log p}{p} \le \log y + E + \frac{1}{2 \log y}.
$$
Setting $y = y_0$, Lemma \ref{lem_logpoverp} follows.
\end{proof}

\begin{lem}\label{lem_1overpsigma}
Let $y \ge y_0$.  Then, for $\sigma$ as above,
$$
\sum_{p \le y} \frac{1}{p^\sigma} < \log \log y + 1.0859.
$$
\end{lem}
\begin{proof}
We first note the helpful fact that
\begin{equation}\label{psigmaineq}
p^{\sigma} = p \exp(- \log p/(2 \log y)) \ge p \exp\left(-1/2\right).
\end{equation}

The idea in this lemma is to bound the difference between $1/p$ and $1/p^{\sigma}$. Let $f(y) = p^{-y}$; then we want to estimate $f(\sigma) - f(1)$. Since $f'(y) = -(\log p) f(y)$, it follows from the mean value theorem that
$f(\sigma)-f(1) = (1-\sigma) (\log p)/p^{y}$
for some $y$ between $\sigma$ and 1. Since $\sigma = 1-1/(2 \log y)$, this gives
\begin{align*}
f(\sigma)-f(1) & \le \frac{1}{2 \log y} \left( \frac{\log p}{p^{\sigma}} \right)\\
& \le \frac{e^{1/2}}{ 2\log y} \left( \frac{\log p}{p} \right),
\end{align*}
by \eqref{psigmaineq}.

Hence, by Lemma \ref{lem_logpoverp}, we have
$$
\sum_{p \le y} \left(\frac{1}{p^{\sigma}} - \frac{1}{p}\right) \le \frac{e^{1/2}}{2\log y} \sum_{p \le y} \frac{ \log p}{p} \le \frac{e^{1/2}}{2} \left( \frac{\log y - 1.3325}{\log y} \right) < 0.8244.
$$
So this means that
$$
\sum_{p \le y} \frac{1}{p^{\sigma}} \le 1 + \sum_{p \le y} \frac{1}{p} \le \log \log y + 1.0859,
$$
by \eqref{eq3}.
\end{proof}

\begin{lem}\label{lem_zeta2} For $y \ge y_0$ and $\sigma$ as above, we have
\begin{equation}\label{zeta2like}
\sum_{p}\frac{1}{p^{\sigma}(p^{\sigma}-1)} < .7734.
\end{equation}
\end{lem}

\begin{proof}
Since $y \ge y_0$, it follows that $\sigma > .9999$.  We will bound \eqref{zeta2like} by explicit computation and analytic methods.  Using 500000 as a useful place to split our computation, we write
\begin{displaymath}
\begin{aligned}
\sum_{p}\frac{1}{p^{\sigma}(p^{\sigma}-1)} & \le \sum_{p \le 500000}\frac{1}{p^{\sigma}(p^{\sigma}-1)} + \sum_{500000 < p < \infty }\frac{1}{p^{\sigma}(p^{\sigma}-1)}\\
& \le .7733545 + \int_{500000}^\infty  \frac{1}{t^{\sigma}(t^{\sigma}-1)} ~ \textrm{dt}\\
& \le .7733545 + .0000001\\
& \le .7734,
\end{aligned}
\end{displaymath}
by explicit computation.
\end{proof}

We may now prove Lemma \ref{lem_bigproduct}.
\begin{proof}
\begin{displaymath}
\begin{aligned}
\prod_{p \le y} \left( \frac{1}{ 1-1/p^{\sigma} } \right) & \le  \exp\left(\sum_{p \le y}\left( \frac{1}{p^{\sigma}} + \frac{1}{p^{2 \sigma}} + \dots\right) \right) \nonumber \\
& = \exp\left(\sum_{p \le y} \frac{1}{p^\sigma} + \frac{1}{p^{\sigma}(p^{\sigma}-1)}\right) \nonumber \\
& =  \exp\left(\sum_{p \le y} \frac{1}{p^\sigma}\right) \cdot \exp \left(\sum_{p \le y}\frac{1}{p^{\sigma}(p^{\sigma}-1)}\right) \nonumber \\
& \le  \exp\left( \log \log y + 1.0859 \right) \exp(.7734) \label{1overpsig} \nonumber \\
& \le 6.4193 \log y,
\end{aligned}
\end{displaymath}
by Lemmas \ref{lem_zeta2} and \ref{lem_1overpsigma}.\end{proof}

We now show that our choice of $x \ge x_0$ is ``sufficiently large'' to adapt Pomerance's argument \cite{Pom2}.  Let
\begin{displaymath}
c = 1 - \left(\log x\right)^{-1/6}\left(\log \log x\right)^{-1}.
\end{displaymath}
For $x \ge x_0$, it is clear that
\begin{displaymath}
c \ge c_0 = 1 - \left(10^6\right)^{-1/6} \left(\log 10^6\right)^{-1} > .99276.
\end{displaymath}

\begin{lem} \label{lem1}
For any $c \ge c_0$ above, we have
\begin{displaymath}
	\sum_{k \ge 2} \frac{1}{k^c \left(k^c - 1\right)} \le 1.0225.
\end{displaymath}
\end{lem}

\begin{proof}
First, we split the sum into two parts, i.e.,
\begin{displaymath}
	\sum_{k \ge 2} \frac{1}{k^c \left(k^c - 1\right)} = \sum_{2 \le k < 10^6} \frac{1}{k^c \left(k^c - 1\right)} + \sum_{k \ge 10^6} \frac{1}{k^c \left(k^c - 1\right)}.
\end{displaymath}
For any value of $k \ge 10^6$, we have $k^c < 1.000002(k^c-1)$, so the second sum can be bounded by
\begin{displaymath}
	  \sum_{k \ge 10^6} \frac{1}{k^c \left(k^c - 1\right)} \le 1.000002 \sum_{k \ge 10^6} \frac{1}{k^{2c}} \le 1.000002 \int_{10^6-1}^{\infty} \frac{dt}{t^{2c}} \le .0000013.
\end{displaymath}
A quick computation shows that
\begin{displaymath}
	\sum_{2 \le k < 10^6} \frac{1}{k^c \left(k^c - 1\right)} \le 1.02247315,
\end{displaymath}
and adding these two together completes the proof of the lemma.
\end{proof}

\begin{lem} \label{lem2}
For any $c \ge c_0$, we have
\begin{displaymath}
	\sum_{p \ge 2} \frac{1}{p^c \left(p^c - 1\right)} \le .7877.
\end{displaymath}
\end{lem}
\begin{proof}
As before, we split the sum into two parts, i.e.,
\begin{displaymath}
	\sum_{p \ge 2} \frac{1}{p^c \left(p^c - 1\right)} = \sum_{2 \le p < 10^6} \frac{1}{p^c \left(p^c - 1\right)} + \sum_{p \ge 10^6} \frac{1}{p^c \left(p^c - 1\right)}.
\end{displaymath}
Once again, we use an explicit computation to show that
\begin{displaymath}
	\sum_{2 \le p \le 10^6} \frac{1}{p^c \left(p^c - 1\right)} \le .7876817684,
\end{displaymath}
and the second sum can be bounded by
\begin{displaymath}
	  \sum_{p \ge 10^6} \frac{1}{p^c \left(p^c - 1\right)} \le .0000013
\end{displaymath}
as in the previous lemma.  Adding these two together completes the proof of the lemma.
\end{proof}

In section \ref{sec4}, we will use the function $\ell = \exp\left( \left(\log x \right)^{1/6} \right)$, which is referenced in the following lemma.

\begin{lem} \label{lem3}
For any $x \ge x_0$ and $c = 1 - \left( \log \log x \right)^{-1} \left(\log x\right)^{-1/6}$, we have
\begin{displaymath}
	\sum_{p \le \ell^4 } \frac{1}{p^c} \le \frac{1}{6} \log \log x + 2.0346.
\end{displaymath}
\end{lem}
\begin{proof}
First, note that
\begin{displaymath}
\begin{aligned}
	p^c = p^{1 - \left( \log \log x \right)^{-1}\left(\log x\right)^{-1/6}} &= p \exp\left( - \frac{\log p}{\left(\log x\right)^{1/6}\left( \log \log x \right)} \right) \\
	&\ge p \exp\left( - \frac{4\left(\log x\right)^{1/6}}{\left(\log x\right)^{1/6}\left( \log \log x \right)} \right) \\
	&= p \exp\left( - \frac{4}{\left( \log \log x \right)} \right) \\
	&\ge .7486p,
	\end{aligned}
\end{displaymath}
since $p \le \ell^4$ and $x \ge e^{10^6}$.  Then, note that the mean value theorem says that, with $f(y) = p^{-y}$ and $f'(y) = -(\log p)/p^y$,
\begin{displaymath}
	f(c) - f(1) \le \frac{(1-c)(\log p)}{p^c}.
\end{displaymath}
Applying both of these facts, we see that
\begin{displaymath}
	\sum_{p \le \ell^4} \left( \frac{1}{p^c} - \frac{1}{p} \right) \le 1.3358 (1-c) \sum_{p \le \ell^4} \frac{\log p}{p}.
\end{displaymath}
By Lemma \ref{lem_logpoverp}, this can be bounded by
\begin{displaymath}
\begin{aligned}
	(1-c) 1.3358 \left( \log \ell^4 - 1.3325\right) &= \frac{ 1.3358 \left(4\left(\log x\right)^{1/6} - 1.3325 \right)}{\left( \log x \right)^{1/6} \left( \log \log x \right)} \\
	&\le \frac{5.3432}{\log \log x} \\
	&\le .3868.
	\end{aligned}
\end{displaymath}
Then, we have
\begin{displaymath}
\begin{aligned}
	\sum_{p \le \ell^4} \frac{1}{p^c} &= \sum_{p \le \ell^4} \frac{1}{p} + \sum_{p \le \ell^4} \left( \frac{1}{p^c} - \frac{1}{p} \right) \\
	&\le \log \log \ell^4 + .3868 + .2615 \\
	&\le \frac{1}{6} \log \log x + 2.0346.
	\end{aligned}
\end{displaymath}
by \eqref{eq3}.
\end{proof}

In \cite{Pom2}, Pomerance uses a function $N(s)$ which, for
\begin{displaymath}
	\psi(m) = \prod_{q^a || m} \left(q+1\right)q^{a-1},
\end{displaymath} 
is the number of $m$ with $\psi(m) \le s$ and $P(\psi(m))$ bounded by a somewhat small function of $x$.  His bound applies only to ``large $x$'', so we must verify that our choice of $x \ge x_0$ is large enough.  Pomerance showed that
\begin{displaymath}
	N(s) \le 2s^c \mathop{\prod_{P(k) < \ell^4}}_{k > 1} \left( 1 - k^{-c} \right)^{-1}.
\end{displaymath}
We now follow \cite{Pom2}, making the bounds explicit as we proceed.

\begin{equation}
\mathop{\prod_{P(k) < \ell^4}}_{k > 1} \left( 1 - k^{-c} \right)^{-1} = \mathop{\prod_{P(k) < \ell^4}}_{k > 1} \left( 1 + \frac{1}{k^c} + \frac{1}{k^{2c}} + \cdots \right).
\label{eqnw}
\end{equation}
Since $1 + t \le e^t$, we have (for $t = \frac{1}{k^c} + \frac{1}{k^{2c}}+ \cdots$)
\begin{displaymath}
	1 + \frac{1}{k^c} + \frac{1}{k^{2c}} + \cdots \le \exp\left( \frac{1}{k^c} + \frac{1}{k^{2c}} + \cdots \right).
\end{displaymath}
Putting this into \eqref{eqnw} gives
\begin{equation} \label{eqnw1}
	\mathop{\prod_{P(k) < \ell^4}}_{k > 1} \left( 1 - k^{-c} \right)^{-1} \le \exp\left( \mathop{\sum_{P(k) < \ell^4}}_{k > 1} k^{-c} + \mathop{\sum_{P(k) < \ell^4}}_{k > 1} \left( k^{-2c} + k^{-3c} + \cdots \right) \right).
\end{equation}
We may bound
\begin{displaymath}
	\mathop{\sum_{P(k) < \ell^4}}_{k > 1} \left( k^{-2c} + k^{-3c} + \cdots \right) \le \sum_{k > 1} \frac{1}{k^c\left(k^c-1\right)} \le 1.0225,
\end{displaymath}
by Lemma \ref{lem1}.  Setting aside this secondary term, the remaining term in the exponential in \eqref{eqnw1} satisfies
\begin{displaymath}
\mathop{\sum_{P(k) < \ell^4}}_{k > 1} k^{-c} = \prod_{p < \ell^4} \left(1 - p^{-c}\right)^{-1}.
\end{displaymath}
Bounding this product by the same argument as in \eqref{eqnw} and \eqref{eqnw1}, we have
\begin{displaymath}
\begin{aligned}
\prod_{p < \ell^4} \left(1 - p^{-c}\right)^{-1}	&\le \exp\left( \sum_{p < \ell^4} p^{-c} + \sum_{p < \ell^4} \frac{1}{p^c(p^c-1)} \right) \\
&\le \exp\left( .7877 \right) \exp\left( \sum_{p < \ell^4} p^{-c} \right),
\end{aligned}
\end{displaymath}
by Lemma \ref{lem2}. An application of Lemma \ref{lem3} shows that this product is bounded by
\begin{displaymath}
	\prod_{p < \ell^4} \left(1 - p^{-c}\right)^{-1}	\le \exp\left( \frac{1}{6} \log \log x + 2.8223 \right) \le \exp(2.8223) \left(\log x\right)^{1/6}.
\end{displaymath}
Thus,
\begin{displaymath}
	\mathop{\prod_{P(k) < \ell^4}}_{k > 1} \left( 1 - k^{-c} \right)^{-1} \le \exp\left( 1.0225 + \exp(2.8223) \left(\log x\right)^{1/6} \right).
\end{displaymath}
We will need the bound in \eqref{eqnw} to hold for
\begin{displaymath}
	s \ge L = \exp\left( 0.1882 \left(\log x\right)^{2/3} \log \log x \right).
\end{displaymath}
If $s \ge L$, the argument above shows that the inequality
\begin{displaymath}
	\begin{aligned}
	\frac{N(s)}{s} &\le 2s^{c-1} \exp\left( 1.0225 + \exp(2.8223) \left(\log x\right)^{1/6} \right) \\
	&= 2 \exp\left( (c-1) \log s + 1.0225 + \exp(2.8223) \left(\log x\right)^{1/6} \right) \\
	&\le 2 \exp\left( -.1882 \left(\log x\right)^{1/2} + 1.0225 + \exp(2.8223) \left(\log x\right)^{1/6} \right) \\
	&\le \frac{2e^{1.0225}}{\ell^2}
	\end{aligned}
\end{displaymath}
holds for $x$ such that
\begin{equation} \label{probeq}
	2 + \exp(2.8223) \le .1882 \left( \log x \right)^{1/3},
\end{equation}
which, in turn, holds true for $x \ge x_0$.

Lastly, we will make use of the fact that for $k \ge 2$,
\begin{equation} \label{eqextra}
	\sum_{n \ge k} \frac{1}{n^2} < \frac{1}{k-1}.
\end{equation}

\section{Large Amicable Numbers} \label{sec4}

\subsection{Outline of the Proof} \label{subsec1}

Pomerance's proof of \eqref{eq1} rests on an argument that five different types of numbers do not contribute much to the sum defining $P$, with a careful count of amicable numbers among the remaining integers up to $x$.  For the sake of brevity, we include only the statement of each assumption, and the resulting bound which must be made explicit to bound the value of $P$.  We make a similar argument here, though we bound $A(x)$ by a much smaller function so that our bound will apply for smaller $x$.  Following the notation of \cite{Pom2}, we will use our previously defined functions
\begin{equation} \label{Choices}
	\ell = \exp\left( \left( \log x \right)^{1/6} \right) \,\, \textrm{and} \,\, L = \exp\left( .1882 \left( \log x \right)^{2/3} \log \log x \right).
\end{equation}
Note that $\ell$ is smaller and $L$ is slightly larger than their equivalents in \cite{Pom2}.  It will also be useful to have $z = 2 x \log \log x$.

Call an integer $n$ \textit{admissible} if it satisfies each of the five conditions given in \cite{Pom2}.  We summarize them here.
\begin{enumerate}
	\item[(i)] The largest prime factor of $n$ and $s(n)$ are both at least $L^2$.
	
	\item[(ii)] If $k^s$ divides $n$ or $s(n)$ with $s \ge 2$, then $k^s < \ell^3$.
	
	\item[(iii)] If $p \mid \gcd\left(n, \sigma(n)\right)$, then $p < \ell^4$.
	
	\item[(iv)] The integer $n$ satisfies
	\begin{displaymath}
	\frac{n}{P(n)} \ge L \qquad \textrm{ and } \qquad \frac{s(n)}{P\left(s(n)\right)} \ge L.
\end{displaymath}
	
	\item[(v)] If $m = \frac{n}{P(n)}$ and $m'=\frac{s(n)}{P\left(s(n)\right)}$, then we have $P\left(\sigma(m)\right) \ge \ell^4$ and $P\left(\sigma(m')\right) \ge \ell^4$.
\end{enumerate}
In showing \eqref{eq1}, Pomerance proved the following (see \cite{Pom2} for details).

\begin{thm}[Pomerance]
The set of integers $n \le x$ which are not admissible is $\displaystyle O\left( x/\exp\left(\left(\log x\right)^{1/3}\right) \right)$.  Specifically, the number of integers $n \le x$ failing conditions (i)-(v) above are bounded by the functions
\begin{enumerate}
	\item[(i)] $\displaystyle \psi\left( z, L^2 \right)$,

\item[(ii)] $\displaystyle z \sum_{k^s \ge \ell^3} k^{-s}$,

\item[(iii)] $\displaystyle \sum_{p \ge \ell^4} \mathop{\sum_{q \equiv -1 (p)}}_{q \le x/p} \frac{x}{pq}$,

\item[(iv)] $\displaystyle \frac{4x \log \log x}{L}$, and

\item[(v)] $\displaystyle \frac{4e^{1.0225}x \log \log x\left(1 + 2\log \log x\right)}{\ell^2} \sum_{p \le x/L} \frac{1}{p}$,
\end{enumerate}
respectively.
\end{thm}

Unfortunately, in \cite{Pom2}, the bounds in (v) are only valid for very large $x$.  By choosing the values of $\ell$ and $L$ in \eqref{Choices}, which are much smaller and slightly larger, respectively, than their counterparts in \cite{Pom2}, we can make the argument apply for all $x \ge x_0 = \exp\left(10^6\right)$.  In making these bounds explicit, we bound the quantity in (i) by the function $C_1 x/\ell$, the function in (ii) by $C_2 x/\ell$, and so on.

Let $S$ denote the set of all admissible integers.  Pomerance demonstrated that the count of amicable numbers in $S$ can be bounded above by
\begin{equation} \label{eq2}
	\sum_{r \ge \ell^4} \mathop{\sum_{q \equiv -1(r)}}_{q \le x} \mathop{\sum_{m \equiv 0 (q)}}_{m \le x} \mathop{\sum_{q' \equiv -1 (r)}}_{q' \le z} \frac{2x \log \log x}{q'm} \le \frac{C_6x}{\ell}
\end{equation}
for some constant $C_6$ (here $(n)$ should be taken to mean $\textrm{mod } n$).

Then, it is clear that we may take $C = C_1 + C_2 + \cdots + C_6$ in the bound
\begin{equation}
A(x) \le C \frac{x}{\exp\left( \left(\log x \right)^{1/6}\right)}
\label{eq:updatedA(x)}
\end{equation}
for $x \ge x_0$, and so it remains to find explicit values for each of these constants.

\subsection{Evaluating the constants}

We address each of the constants given in subsection \ref{subsec1}.

\begin{enumerate}

\item[$C_1$:]

We use an explicit version of the method of Rankin to bound $\psi(z,L^2)$.  Given that $x \ge x_0$, we also have that $L^2 \ge y_0$.  Rankin's method is based on the observation that, for any $\sigma > 0, x \ge 1,$ and $y \ge 2$,
\begin{equation}\label{Rankin1}
\begin{aligned}
\psi(x,y) &\le \sum_{\substack{n \ge 1 \\ P(n) \le q}} \left(\frac{x}{n}\right)^\sigma = x^\sigma \sum_{\substack{n \ge 1 \\ P(n) \le q}} \frac{1}{n^\sigma} \\
&= x^\sigma \prod_{p \le y} \left( 1 - \frac{1}{p^\sigma}\right)^{-1}.
\end{aligned}
\end{equation}

From Lemma \ref{lem_bigproduct} is follows that
$$
\psi(x,y) \le 6.4193 \left(x^\sigma\right) (\log y),
$$
and, therefore, for $\sigma = 1 - 1/(2 \log y)$, that
\begin{displaymath}
\begin{aligned}
\psi(z,L^2) &\le 6.4193 \cdot z^\sigma  \cdot .3764 \left(\log x\right)^{2/3} \log \log x \\
&\le 2.4163 \cdot x^{\sigma} \cdot 2^{\sigma} \left(\log \log x\right)^{1+\sigma} \cdot \left(\log x\right)^{2/3}.
\end{aligned}
\end{displaymath}

Thus, we want to find $C_1$ such that
\begin{displaymath}
	\frac{4.8325 \cdot x \cdot \left(\log \log x\right)^2 \cdot \left(\log x\right)^{2/3}}{z^{1/(.7528 \left(\log x\right)^{2/3} \log \log x) }} \le C_1 \frac{x}{\ell}.
\end{displaymath}
Plugging in $x_0$, we may take $C_1 = 13553617.97$.\\

It turns out that the other constants are negligible compared to $C_1$, so we will sacrifice some sharpness in bounds for space in what follows.\\

\item[$C_2$:] Note that we can bound
\begin{displaymath}
\begin{aligned}
	z \mathop{\sum_{k^m \ge \ell^3}}_{m \ge 2} \frac{1}{k^m} &< z \mathop{\sum_{k \ge \ell^{\frac{3}{2}}}}\left(\frac{1}{k^2} + \frac{1}{k^3} + \cdots \right) = z \mathop{\sum_{k \ge \ell^{\frac{3}{2}}}}\frac{1}{k^2 - k} \\
	&= z \mathop{\sum_{k \ge \ell^{\frac{3}{2}}}} \frac{1}{k(k-1)} < 1.0001z \mathop{\sum_{k \ge \ell^{\frac{3}{2}}}}\frac{1}{k^2} < \frac{1.0001z}{\ell^{\frac{3}{2}}-1}
	\end{aligned}
	\end{displaymath}
	for $x \ge x_0$.   Now we need
\begin{displaymath}
\frac{1.0001z}{\ell^{\frac{3}{2}}-1} < C_2 \frac{x}{\ell}.
\end{displaymath}
Plugging in $x = x_0$ into the formula for $\ell$ shows that we may choose $C_2 = .1862$.\\

\item[$C_3$:] A simple application of Lemma \ref{myLem} gives
\begin{displaymath}
	\frac{x}{p} \mathop{\sum_{q \equiv -1 (p)}}_{q \le x/p} \frac{1}{q} \le \frac{x}{p} \left( \frac{4 + 3\log \log \left(x/p\right)}{p} \right).
\end{displaymath}
Now, summing this over $p \ge \ell^4$ gives at most
\begin{displaymath}
	 x(4 + 3\log \log x) \sum_{p \ge \ell^4} \frac{1}{p^2} < \frac{x(4 + 3\log \log x)}{\ell^4 - 1},
\end{displaymath}
by \eqref{eqextra}.  Therefore, we need
\begin{displaymath}
	\frac{x(4 + 3\log \log x)}{\ell^4 - 1} \le \frac{C_3x}{\ell},
\end{displaymath}
which holds for $C_3 > 4.3 \times 10^{-12}$.\\

\item[$C_4$:] Here, we need
\begin{displaymath}
	\frac{4 x \log \log x}{L^2} < \frac{C_4x}{\ell}
\end{displaymath}
or
\begin{displaymath}
\frac{4 \log \log x \exp\left( \left( \log x \right)^{1/6} \right)}{\exp\left( .61\left( \log x \right)^{2/3} \log \log x \right)}	< C_4.
\end{displaymath}
Plugging in $x = x_0$ gives that we may choose any $C_4 > 10^{-17471}$.\\

\item[$C_5$:] We must next compute $C_5$ such that the inequality
\begin{displaymath}
	\frac{4e^{1.0225}x \log \log x\left(1+2\log\log x\right)}{\ell^2} \sum_{p \le x/\ell} \frac{1}{p} \le \frac{C_5x}{\ell}
\end{displaymath}
holds for all $x \ge x_0$.  By \eqref{eq3}, this requires
\begin{displaymath}
	\frac{4e^{1.0225} \log \log x \left(1+2\log\log x\right)}{\exp\left( \left( \log x \right)^{1/6} \right)} \left( \log \log x + .2615 \right) < C_5.
\end{displaymath}
For $x \ge x_0$, we may take $C_5 = 2.8117$.\\

\item[$C_6$:] We have
\begin{equation} \label{eq5}
	\sum_{r \ge \ell^4} \mathop{\sum_{q \equiv -1(r)}}_{q \le x} \mathop{\sum_{m \equiv 0 (q)}}_{m \le x} \mathop{\sum_{q' \equiv -1 (r)}}_{q' \le z} \frac{2x \log \log x}{q'm} = z \sum_r \sum_q \sum_m \frac{1}{m} \sum_{q'} \frac{1}{q'}.
	\end{equation}
	Applying Lemma \ref{myLem}, we find \eqref{eq5} is bounded by
\begin{equation} \label{eq7}
z (4 + 3\log \log z)  \sum_r \sum_q \sum_m \frac{1}{m}.
\end{equation}
Now, since $\log (x/q) < \log x$, we can bound
\begin{displaymath}
	\mathop{\sum_{m \equiv 0 (q)}}_{m \le x} \frac{1}{m} = \sum_{n \le x/q} \frac{1}{nq} = \frac{1}{q} \sum_{n \le x/q} \frac{1}{n} \le \frac{\log x}{q}.
\end{displaymath}
Using this and another application of Lemma \ref{myLem}, the iterated sum in \eqref{eq7} above can be bounded by
\begin{equation} \label{eq9}
z (4 + 3\log \log z) (\log x) (4 + 3\log \log x) \sum_r \frac{1}{r^2}.
\end{equation}
Lastly, \eqref{eqextra} shows that
	\begin{displaymath}
	\sum_{r \ge \ell^4} \frac{1}{r^2} < \frac{1}{\ell^4 - 1},
\end{displaymath}
and so, putting this into \eqref{eq9}, we need
	\begin{equation} \label{eq11}
\frac{z (4 + 3\log \log z) (\log x) (4 + 3\log \log x)}{\ell^4 - 1} \le C_6 \frac{x}{\ell}.
\end{equation}
A calculation shows that \eqref{eq11} holds for $C_6 = .0054$.
\end{enumerate}

Thus, for $x \ge \exp\left(10^6\right)$, the number of amicable numbers not greater than $x$, $A(x)$, can be bounded as
\begin{displaymath}
A(x) \le 13553620.97 \frac{x}{\exp\left( \left(\log x\right)^{1/6}\right)}.
\end{displaymath}
Then, by partial summation, we have that the third sum in \eqref{split} is bounded by
\begin{displaymath}
	13553620.97 \left( \exp\left(-10\right)+1 \right) \int_{x_0}^{\infty} \frac{dt}{t \exp\left( \left( \log t \right)^{1/6}\right)} \le 654666169.
\end{displaymath}
Combined with the results of sections \ref{sec1} and \ref{sec3}, this proves Theorem \ref{Thm1}.\\

\section{Acknowledgements}

The authors would like to thank Carl Pomerance for many helpful suggestions, as well as for his expert guidance over the years. Thanks to Paul Pollack for his help in making some of the bounds in \cite{Pom2} explicit in Section \ref{sec2}, and Kevin Ford for a number of helpful ideas.  We also thank an anonymous referee for many suggested improvements in the writing of this paper.

\end{document}